\documentclass[a4paper,12pt]{amsart}
\usepackage{mathrsfs}
\usepackage{cases}
\usepackage{amsfonts}
\usepackage{graphicx}
\usepackage{amsmath}
\usepackage{amssymb}
\usepackage[all]{xypic}
\usepackage[all]{xy}
\usepackage{color}
\usepackage{colordvi}
\usepackage{multicol}
\begin{document}
\input xy
\xyoption{all}

\renewcommand{\mod}{\operatorname{mod}\nolimits}
\newcommand{\proj}{\operatorname{proj.}\nolimits}
\newcommand{\inj}{\operatorname{inj.}\nolimits}
\newcommand{\rad}{\operatorname{rad}\nolimits}
\newcommand{\soc}{\operatorname{soc}\nolimits}
\newcommand{\ind}{\operatorname{inj.dim}\nolimits}
\newcommand{\Ginj}{\operatorname{Ginj}\nolimits}
\newcommand{\Mod}{\operatorname{Mod}\nolimits}
\newcommand{\R}{\operatorname{R}\nolimits}
\newcommand{\End}{\operatorname{End}\nolimits}
\newcommand{\ob}{\operatorname{Ob}\nolimits}
\newcommand{\Ht}{\operatorname{Ht}\nolimits}
\newcommand{\cone}{\operatorname{cone}\nolimits}
\newcommand{\rep}{\operatorname{rep}\nolimits}
\newcommand{\Ext}{\operatorname{Ext}\nolimits}
\newcommand{\Tor}{\operatorname{Tor}\nolimits}
\newcommand{\Hom}{\operatorname{Hom}\nolimits}
\newcommand{\Pic}{\operatorname{Pic}\nolimits}
\newcommand{\Span}{\operatorname{Span}\nolimits}
\newcommand{\Coker}{\operatorname{Coker}\nolimits}
\newcommand{\Div}{\operatorname{Div}\nolimits}
\newcommand{\rank}{\operatorname{rank}\nolimits}
\newcommand{\gldim}{\operatorname{gldim}\nolimits}
\newcommand{\Gproj}{\operatorname{Gproj}\nolimits}
\newcommand{\Char}{\operatorname{char}\nolimits}
\newcommand{\Len}{\operatorname{Length}\nolimits}
\newcommand{\RHom}{\operatorname{RHom}\nolimits}
\renewcommand{\deg}{\operatorname{deg}\nolimits}
\renewcommand{\Im}{\operatorname{Im}\nolimits}
\newcommand{\Ker}{\operatorname{Ker}\nolimits}
\newcommand{\Coh}{\operatorname{Coh}\nolimits}
\newcommand{\Id}{\operatorname{Id}\nolimits}
\newcommand{\Qcoh}{\operatorname{Qch}\nolimits}
\newcommand{\CM}{\operatorname{CM}\nolimits}
\newcommand{\Cp}{\operatorname{Cp}\nolimits}
\newcommand{\coker}{\operatorname{Coker}\nolimits}
\renewcommand{\dim}{\operatorname{dim}\nolimits}
\renewcommand{\div}{\operatorname{div}\nolimits}
\newcommand{\Ab}{{\operatorname{Ab}\nolimits}}
\renewcommand{\Vec}{{\operatorname{Vec}\nolimits}}
\newcommand{\pd}{\operatorname{proj.dim}\nolimits}
\newcommand{\id}{\operatorname{inj.dim}\nolimits}
\newcommand{\Gd}{\operatorname{G.dim}\nolimits}
\newcommand{\sdim}{\operatorname{sdim}\nolimits}
\newcommand{\add}{\operatorname{add}\nolimits}
\newcommand{\pr}{\operatorname{pr}\nolimits}
\newcommand{\oR}{\operatorname{R}\nolimits}
\newcommand{\oL}{\operatorname{L}\nolimits}
\newcommand{\Perf}{{\mathfrak Perf}}
\newcommand{\cc}{{\mathcal C}}
\newcommand{\ce}{{\mathcal E}}
\newcommand{\cs}{{\mathcal S}}
\newcommand{\cf}{{\mathcal F}}
\newcommand{\cx}{{\mathcal X}}
\newcommand{\ct}{{\mathcal T}}
\newcommand{\cu}{{\mathcal U}}
\newcommand{\cv}{{\mathcal V}}
\newcommand{\cn}{{\mathcal N}}
\newcommand{\ch}{{\mathcal H}}
\newcommand{\ca}{{\mathcal A}}
\newcommand{\cb}{{\mathcal B}}
\newcommand{\ci}{{\mathcal I}}
\newcommand{\cj}{{\mathcal J}}
\newcommand{\cm}{{\mathcal M}}
\newcommand{\cp}{{\mathcal P}}
\newcommand{\cg}{{\mathcal G}}
\newcommand{\cw}{{\mathcal W}}
\newcommand{\co}{{\mathcal O}}
\newcommand{\cd}{{\mathcal D}}
\newcommand{\ck}{{\mathcal K}}
\newcommand{\calr}{{\mathcal R}}
\newcommand{\ol}{\overline}
\newcommand{\ul}{\underline}
\newcommand{\st}{[1]}
\newcommand{\ow}{\widetilde}
\renewcommand{\P}{\mathbf{P}}
\newcommand{\pic}{\operatorname{Pic}\nolimits}
\newcommand{\Spec}{\operatorname{Spec}\nolimits}
\newtheorem{theorem}{Theorem}[section]
\newtheorem{acknowledgement}[theorem]{Acknowledgement}
\newtheorem{algorithm}[theorem]{Algorithm}
\newtheorem{axiom}[theorem]{Axiom}
\newtheorem{case}[theorem]{Case}
\newtheorem{claim}[theorem]{Claim}
\newtheorem{conclusion}[theorem]{Conclusion}
\newtheorem{condition}[theorem]{Condition}
\newtheorem{conjecture}[theorem]{Conjecture}
\newtheorem{construction}[theorem]{Construction}
\newtheorem{corollary}[theorem]{Corollary}
\newtheorem{criterion}[theorem]{Criterion}
\newtheorem{definition}[theorem]{Definition}
\newtheorem{example}[theorem]{Example}
\newtheorem{exercise}[theorem]{Exercise}
\newtheorem{lemma}[theorem]{Lemma}
\newtheorem{notation}[theorem]{Notation}
\newtheorem{problem}[theorem]{Problem}
\newtheorem{proposition}[theorem]{Proposition}
\newtheorem{remark}[theorem]{Remark}
\newtheorem{solution}[theorem]{Solution}
\newtheorem{summary}[theorem]{Summary}
\newtheorem*{thm}{Theorem}

\def \bp{{\mathbf p}}
\def \bA{{\mathbf A}}
\def \bL{{\mathbf L}}
\def \bF{{\mathbf F}}
\def \bS{{\mathbf S}}
\def \bC{{\mathbf C}}

\def \Z{{\Bbb Z}}
\def \F{{\Bbb F}}
\def \C{{\Bbb C}}
\def \N{{\Bbb N}}
\def \Q{{\Bbb Q}}
\def \G{{\Bbb G}}
\def \P{{\Bbb P}}
\def \K{{\Bbb K}}
\def \E{{\Bbb E}}
\def \A{{\Bbb A}}
\def \BH{{\Bbb H}}
\def \T{{\Bbb T}}
\newcommand {\lu}[1]{\textcolor{red}{$\clubsuit$: #1}}

\title[]{singularity categories of skewed-gentle algebras}

\author[Chen]{Xinhong Chen}
\address{Department of Mathematics, Southwest Jiaotong University, Chengdu 610031, P.R.China}
\email{chenxinhong@swjtu.edu.cn}

\author[Lu]{Ming Lu$^\dag$}
\address{Department of Mathematics, Sichuan University, Chengdu 610064, P.R.China}
\email{luming@scu.edu.cn}
\thanks{$^\dag$ Corresponding author}


\begin{abstract}
Let $K$ be an algebraically closed field. Let $(Q,Sp,I)$ be a skewed-gentle triple, $(Q^{sg},I^{sg})$ and $(Q^g,I^{g})$ be its corresponding skewed-gentle pair and associated gentle pair respectively. It proves that the skewed-gentle algebra $KQ^{sg}/\langle I^{sg}\rangle$ is singularity equivalent to
$KQ/\langle I\rangle$. Moreover, we use $(Q,Sp,I)$ to describe the singularity category of $KQ^g/\langle I^g\rangle$. As a corollary, we get that $\gldim KQ^{sg}/\langle I^{sg}\rangle<\infty$ if and only if
$\gldim KQ/\langle I\rangle<\infty$ if and only if $\gldim KQ^{g}/\langle I^{g}\rangle<\infty$.
\end{abstract}

\maketitle

\section{Introduction}
The singularity category of an algebra is defined to be the Verdier quotient of the bounded derived category with respect to the thick subcategory formed by complexes isomorphic to bounded complexes of finitely generated projective modules \cite{Bu}, see also \cite{Ha1}. Recently, Orlov's global version \cite{Or1} attracted a lot of interest in algebraic geometry and theoretical physics. In particular, the singularity category measures the homological singularity of an algebra \cite{Ha1}: the algebra has finite global dimension if and only if its singularity category is trivial.

A fundamental result of Buchweitz \cite{Bu} and Happel \cite{Ha1} states that for a Gorenstein algebra $A$, the singularity category is triangle equivalent to the stable category of (maximal) Cohen-Macaulay (also called Gorenstein projective) $A$-modules, which generalized Rickard's result \cite{Ri} on self-injective algebras.

As an important class of Gorenstein algebras \cite{GR}, gentle algebras were introduced in \cite{AS} as appropriate context for the investigation of algebras derived equivalent to hereditary algebras of type $\tilde{\A}_n$. Moreover, many important algebras are gentle, such as tilted algebras of type $\A_n$, algebras derived equivalent to $\A_n$-configurations of projective lines \cite{Bur} and also the cluster-tilted algebras of type $\A_n$, $\tilde{\A}_n$ \cite{ABCP}. As a generalization of gentle algebras, skewed-gentle algebras were introduced by Gei{\ss} and de la Pe\~{na} \cite{GdlP}, and they also proved that a skewed-gentle algebra is Morita equivalent to a skew-group algebra of a gentle algebra (which is called \emph{the associated gentle algebra} in this note) with a group of order two. In this way, the skewed-gentle algebras and gentle algebras share many common properties, such as Gorenstein property \cite{GR}, and the derived categories \cite{BM,BMM}, etc..
For singularity categories of gentle algebras, Kalck determines their singularity category by finite products of $n$-cluster categories of type $\A_1$ \cite{Ka}.

The aim of this note is to describe the singularity categories for skewed-gentle algebras, following Kalck's work \cite{Ka}.
In order to state our main results, we need introduce some notations. Let $(Q,Sp,I)$ be a skewed-gentle triple, $(Q^{sg},I^{sg})$ its corresponding skewed-gentle pair, and $(Q^g, I^{g})$ its associated gentle pair. Inspired by \cite{Chen2}, which shows a certain homological epimorphism between two algebras induces a triangle equivalence between their singularity categories, we prove that there is a morphism of this type between the skewed-gentle algebra $KQ^{sg}/\langle I^{sg}\rangle$ and the gentle algebra $KQ/\langle I\rangle$, and so they are singularity equivalent, see Theorem \ref{main theorem}. Besides, with the help of \cite{Ka}, we also use $(Q,Sp,I)$ to describe the singularity categories of the associated gentle algebra $KQ^{g}/\langle I^{g}\rangle$, and then get the relation between it and the singularity category of $KQ^{sg}/\langle I^{sg}\rangle$, see Theorem \ref{main theorem 2}. As a direct corollary of them, we get that
the global dimension of $KQ^{sg}/\langle I^{sg}\rangle$ is finite if and only if the global dimension of $KQ^g/\langle I^{g}\rangle$ is finite, if and only if the global dimension of $KQ/\langle I\rangle$ is finite, without any restriction on the characteristic of the field $K$, see Corollary \ref{corollary 2}.

\vspace{0.2cm} \noindent{\bf Acknowledgments.}
This work is inspired by some discussions with Changjian Fu. The first author(X. Chen) was supported by the Fundamental Research Funds for the Central Universities A0920502051411-45.
The second author(M. Lu) was supported by the National Natural Science Foundation of China (No. 11401401).

\section{Preliminaries}
Throughout this note, we always assume that $K$ is an algebraically closed field. For any finite set $S$, we denote by $\sharp (S)$ the number of the elements in $S$. For any algebra $A$, we denote by $\gldim A$ its global dimension.

Let $Q$ be a quiver and $\langle I\rangle$ an admissible ideal in the path algebra $KQ$ which is generated by a set of relations $I$. Denote by $(Q,I)$ the \emph{associated bound quiver}. For any arrow $\alpha$ in $Q$ we denote by $s(\alpha)$ its starting vertex and by $t(\alpha)$ its ending vertex. An \emph{oriented path} (or path for short) $p$ in $Q$ is a sequence $p=\alpha_1\alpha_2\dots \alpha_r$ of arrows $\alpha_i$ such that $t(\alpha_i)=s(\alpha_{i-1})$ for all $i=2,\dots,r$. For any two paths $p_1,p_2$ in $(Q,I)$, we denote by $p_1\sim p_2$ the relation of $p_1-p_2\in \langle I\rangle$.
\subsection{Gentle algebras}
We first recall the definition of special biserial algebras and of gentle algebras.
\begin{definition}[\cite{SW}]
The pair $(Q,I)$ is called \emph{special biserial} if it satisfies the following conditions.
\begin{itemize}
\item Each vertex of $Q$ is starting point of at most two arrows, and end point of at most two arrows.
\item For each arrow $\alpha$ in $Q$ there is at most one arrow $\beta$ such that $\alpha\beta\notin I$, and at most one arrow $\gamma$ such that $\gamma\alpha\notin I$.
\end{itemize}
\end{definition}
\begin{definition}[\cite{AS}]
The pair $(Q,I)$ is called \emph{gentle} if it is special biserial and moreover the following holds.
\begin{itemize}
\item The set $I$ is generated by zero-relations of length $2$.
\item For each arrow $\alpha$ in $Q$ there is at most one arrow $\beta$ with $t(\beta)=s(\alpha)$ such that $\alpha\beta\in I$, and at most one arrow $\gamma$ with $s(\gamma)=t(\alpha)$ such that $\gamma\alpha\in I$.
\end{itemize}

\end{definition}
A finite dimensional algebra $A$ is called \emph{special biserial} (resp., \emph{gentle}), if it has a presentation as $A=KQ/\langle I\rangle$ where $(Q,I)$ is special biserial (resp., gentle).

\subsection{Skewed-gentle algebras}

Skewed-gentle algebras were introduced in \cite{GdlP}; for the notion and notation we follow here mostly \cite{BMM}.

Let $(Q,I)$ be a gentle pair. Let $Sp$ be a subset of vertices of the quiver $Q$ and the elements of $Sp$ are called to be \emph{special vertices}, the remaining vertices are called \emph{ordinary}.

For a triple $(Q,Sp,I)$, let us consider the pair $(Q^{sp},I^{sp})$, where $Q_0^{sp}:=Q_0$, $Q_1^{sp}:=Q_1\cup\{\alpha_i|i\in Sp,s(\alpha_i)=t(\alpha_i)=i\}$ and $I^{sp}:=I\cup\{\alpha_i^2|i\in Sp\}$.

\begin{definition}
A triple $(Q,Sp,I)$ as above is called \emph{skewed-gentle} if the corresponding pair $(Q^{sp},I^{sp})$ is gentle.
\end{definition}

For any vertex in a quiver $Q$, its \emph{valency} is defined as the number of arrows attached to it, i.e., the number of incoming arrows plus the number of outgoing arrows(note that in particular any loop contributes twice to the valency).

In fact, Bessenrodt and Holm pointed out that the admissibility of the set $Sp$ of special vertices is both a local as well as a global condition \cite{BH}. Let $v$ be a vertex in the gentle quiver $(Q,I)$; then we can only add a loop at $v$ if $v$ is of valency $1$ or $0$ or if it is of valency $2$ with a zero relation, but not one coming from a loop. Furthermore, for the choice of an admissible set of special vertices we also have to take care of the global condition
that after adding all loops, the pair $(Q^{sp},I^{sp})$ still does not have paths of arbitrary lengths.

\begin{example}\label{example 1}
(a) Let $(Q,I)$ be the bound quiver as the following diagram shows. Then $(Q,I)$ is gentle. In order to make $(Q,Sp,I)$ to be skewed-gentle, the set $Sp$ can only be $\{1\}$, $\{2\}$ or the empty set.
\[ \xymatrix{\circ^1\ar@<0.5ex>[r]^{\alpha} &   \circ^2 \ar@<0.5ex>[l]^{\beta} &I=\{\alpha\beta,\beta\alpha\}
}\]

(b) Let $(Q,I)$ be the bound quiver as the following diagram shows. Then $(Q,I)$ is gentle. In order to make $(Q,Sp,I)$ to be skewed-gentle, the set $Sp$ can only be $\{1\}$, $\{2\}$, $\{3\}$, $\{1,2\}$, $\{2,3\}$, $\{1,3\}$ or the empty set.
\[ \xymatrix{&\circ^3\ar[dl]_\gamma&&\\
\circ^1\ar[rr]_{\alpha} &  & \circ^2 \ar[ul]_{\beta} &I=\{\alpha\gamma,\beta\alpha,\gamma\beta\}
}\]
\end{example}

Let $(Q,Sp,I)$ be a skewed-gentle triple. We associate to each vertex $i\in Q_0$ a set, denoted by $Q_0(i)$, on the following way: If $i$ is an ordinary vertex then $Q_0(i)=\{i\}$; if $i$ is special then $Q_0(i)=\{i^-,i^+\}$.
We denote by $(Q^{sg},I^{sg})$ the pair defined in the following:
\begin{eqnarray*}
&&Q_0^{sg}:=\bigcup_{i\in Q_0}Q_0(i),\\
&&Q_1^{sg}[a,b]:=\{(a,\alpha,b)|\alpha\in Q_1,a\in Q_0(s(\alpha)),b\in Q_0(t(\alpha))\},\\
&&I^{sg}:=\{\sum_{b\in Q_0(s(\alpha))} \lambda_b(b,\alpha,c)(a,\beta,b)|\alpha\beta\in I,a\in Q_0(s(\beta)),c\in Q_0(t(\alpha))\},
\end{eqnarray*}
where $\lambda_b=-1$ if $b=i^-$ for some $i\in Q_0$, and $\lambda_b=1$ otherwise.

Note that the relations in $I^{sg}$ are zero relations or commutative relations.

\begin{definition}[\cite{GdlP}]
A $K$-algebra $A$ is called skewed-gentle, if it is Morita equivalent to a factor algebra $KQ^{sg}/\langle I^{sg}\rangle$, where the triple $(Q,Sp,I)$ is skewed-gentle. The corresponding pair $(Q^{sg},I^{sg})$ is also said to be skewed-gentle.
\end{definition}

\begin{example}\label{example 2}
(a) Keep the notations as in Example \ref{example 1} (a). If $Sp=\{2\}$, then $(Q^{sg},I^{sg})$ is as the following shows.

\[ \xymatrix{\circ^{2^-} \ar@<0.5ex>[r]^{\beta^-} & \circ^1\ar@<0.5ex>[r]^{\alpha^+}\ar@<0.5ex>[l]^{\alpha^-} &   \circ^{2^+} \ar@<0.5ex>[l]^{\beta^+} &I^{sg}=\{\alpha^\pm\beta^\pm,\beta^+\alpha^+-\beta^-\alpha^-\}
}\]

(b) Keep the notations as in Example \ref{example 1} (b). If $Sp=\{3\}$, then $(Q^{sg},I^{sg})$ is as the following shows.
\[ \xymatrix{&\circ^{3^+}\ar[dl]_{\gamma^+}&&\\
\circ^1\ar[rr]_{\alpha} &  & \circ^2 \ar[ul]_{\beta^+} \ar[dl]^{\beta^-} &I^{sg}=\{\alpha\gamma^{\pm},\beta^{\pm}\alpha,\gamma^+\beta^+-\gamma^-\beta^-\}\\
&\circ^{3^-}\ar[ul]^{\gamma^-} &&
}\]
\end{example}

\subsection{Skew-group algebras}
It follows from Gei{\ss} and de la Pe\~{n}a\cite{GdlP} that for any skewed-gentle triple $(Q,Sp,I)$, the corresponding skewed-gentle algebra $KQ^{sg}/\langle I^{sg}\rangle$ is Morita-equivalent to a skew-group
algebra $BG$ in the case of $\mathrm{char} K\neq2$, where $B$ is a gentle algebra and $G$ is a finite group, see also \cite{BMM}.
We recall it in the following precisely.

Let $A$ be a $K$-algebra, and $G$ a finite group acting on $A$ via $K$-linear automorphisms. The \emph{skew-group algebra} $AG$ is the vector space $\oplus_{g\in G}A[g]$ with multiplication induced by
$$a[g]b[h]:=ag(b)[gh].$$

Let $(Q,Sp,I)$ be a skewed-gentle triple. For a given special (resp., ordinary) vertex $i$, let us denote by $Q_0[i]$ the set $\{i\}$ (resp., $\{i^-,i^+\}$). Consider the pair $(Q^g,I^g)$, where $Q_0^g:=\cup_{i\in Q_0}Q_0[i]$, $Q_1^g:=\{\alpha^+,\alpha^-|\alpha\in Q_1\}$,
$$s(\alpha^{\pm}):=\left\{ \begin{array}{cc} s(\alpha)^{\pm}, &\mbox{if } s(\alpha)\notin Sp;\\
s(\alpha), &\mbox{if } s(\alpha)\in Sp,\end{array}
 \right.\quad\quad t(\alpha^{\pm}):=\left\{ \begin{array}{cc} t(\alpha)^{\pm}, &\mbox{if } t(\alpha)\notin Sp;\\
t(\alpha), &\mbox{if } t(\alpha)\in Sp
 \end{array}\right.$$
and $$I^g:=\{\beta^+\alpha^+,\beta^-\alpha^-|\beta\alpha\in I,t(\alpha)\notin Sp\}\bigcup \{\beta^+\alpha^-,\beta^-\alpha^+|\beta\alpha\in I,t(\alpha)\in Sp\}.$$
It follows from \cite{GdlP} that the algebra $B:=KQ^{g}/\langle I^g\rangle$ is gentle. We call $(Q^g,I^g)$ (resp., $ KQ^g/\langle I^g\rangle$) the \emph{associated gentle pair} (resp., \emph{associated gentle algebra}) of $(Q,Sp,I)$ or $KQ^{sg}/\langle I^{sg}\rangle$.

Consider the group $G=\{e,g|g^2=e\}$ which acts on $B$ defined by the rule:
$$g(i^{\pm}):=i^{\mp},\quad g(j):=j,\quad g(\alpha^{\pm}):=\alpha^{\mp}$$
for all $i\in Q_0\setminus Sp$, $j\in Sp$ and $\alpha\in Q_1$. Then we get the skew-group algebra $BG$.
\begin{example}\label{example 3}
(a) Keep the notations as in Example \ref{example 2} (a). If $Sp=\{2\}$, then $(Q^{g},I^{g})$ is as the following shows.

\[ \xymatrix{\circ^{1^+} \ar@<0.5ex>[r]^{\alpha^+} & \circ^2\ar@<0.5ex>[r]^{\beta^-}\ar@<0.5ex>[l]^{\beta^+} &   \circ^{1^-} \ar@<0.5ex>[l]^{\alpha^-} &I^{g}=\{\alpha^+\beta^+,\alpha^-\beta^-,\beta^+\alpha^-,\beta^-\alpha^+\}
}\]

(b) Keep the notations as in Example \ref{example 2} (b). If $Sp=\{3\}$, then $(Q^{g},I^{g})$ is as the following shows.
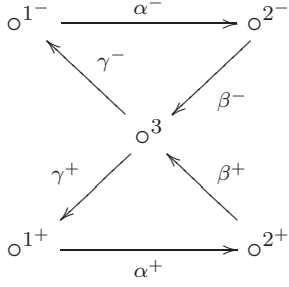
\[ \xymatrix{
\circ^{1^-}\ar[rr]^{\alpha^-}&&\circ^{2^-}\ar[dl]^{\beta^-}\\
&\circ^{3}\ar[dl]_{\gamma^+}\ar[ul]_{\gamma^-}&&I^g=\{\alpha^+\gamma^{+},\alpha^-\gamma^{-},\beta^{+}\alpha^+,\beta^-\alpha^-,\gamma^-\beta^+,
\gamma^+\beta^-\}\\
\circ^{1^+}\ar[rr]_{\alpha^+} &  & \circ^{2^+} \ar[ul]_{\beta^+} &
}\]

\end{example}

\subsection{Singularity categories and Gorenstein algebras}
Let $\Gamma$ be a finite-dimensional $K$-algebra. Let $\mod \Gamma$ be the category of finitely generated left $\Gamma$-modules. By $D=\Hom_K(-,K)$ we denote the standard duality with respect to the ground field. Then $_\Gamma D(\Gamma_\Gamma)$ is an injective cogenerator for $\mod \Gamma$. For an arbitrary $\Gamma$-module $_\Gamma X$, we denote by $\pd_\Gamma X$ (resp. $\id_\Gamma X$) the projective dimension (resp. the injective dimension) of the module $_\Gamma X$. A $\Gamma$-module $G$ is \emph{Gorenstein projective}, if there is an exact sequence $$P^\bullet:\cdots \rightarrow P^{-1}\rightarrow P^0\xrightarrow{d^0}P^1\rightarrow \cdots$$ of projective $\Gamma$-modules, which stays exact under $\Hom_\Gamma(-,\Gamma)$, and such that $G\cong \Ker d^0$. We denote by $\Gproj(\Gamma)$ the subcategory of Gorenstein projective $\Gamma$-modules.

\begin{definition}[\cite{AR1,AR2,Ha1}]
A finite dimensional algebra $\Gamma$ is called a Gorenstein algebra if $\Gamma$ satisfies $\pd_\Gamma D(\Gamma_\Gamma)<\infty$ and $\id_\Gamma \Gamma<\infty$.
\end{definition}

For an algebra $\Gamma$, the \emph{singularity category} of $\Gamma$ is defined to be the quotient category $D_{sg}^b(\Gamma):=D^b(\Gamma)/K^b(\proj \Gamma)$ \cite{Bu,Ha1,Or1}. Note that $D_{sg}^b(\Gamma)$ is zero if and only if $\gldim \Gamma<\infty$ \cite{Ha1}.
For any two algebras, if their singularity categories are equivalent, then we call them to be \emph{singularity equivalent}.

\begin{theorem}[\cite{Bu,Ha1}]
Let $\Gamma$ be a Gorenstein algebra. Then $\Gproj (\Gamma)$ is a Frobenius category with the projective modules as the projective-injective objects. The stable category $\underline{\Gproj}(\Gamma)$ is triangle equivalent to the singularity category $D^b_{sg}(\Gamma)$ of $\Gamma$.
\end{theorem}

Gei{\ss} and Reiten \cite{GR} have shown that gentle algebras are Gorenstein algebras. Since the property of being Gorenstein is also preserved under the skew-group
ring construction with a finite group whose order is invertible in $K$, see \cite{RR,AR2}, Gei{\ss} and Reiten also pointed out that skewed-gentle algebras are Gorenstein algebras in the case of $\mathrm{char}K\neq2$ \cite{GR}.

\section{The first main theorem}

In order to prove the first main result, we describe a construction of matrix algebras which is obtained by X-W Chen in \cite{Chen2}, see also \cite[Section 4]{KN}. Let $A$ be a finite dimensional algebra over a field $K$. Let $_AM$ and $N_A$ be a left and right $A$-module, respectively. Then $M\otimes_KN$ becomes an $A\mbox{-}A$-bimodule. Consider an $A\mbox{-}A$-bimodule monomorphism $\phi:M\otimes_K N\rightarrow A$ such that $\Im\phi$ vanishes both on $M$ and $N$. Note that $\Im \phi\subseteq A$ is an ideal. Then the matrix $\Gamma=\left( \begin{array}{cc} A &M\\ N& K\end{array}\right)$ becomes an associative algebra via the following multiplication
$$\left( \begin{array}{cc} a &m\\ n& \lambda\end{array}\right)\left( \begin{array}{cc} a' &m'\\ n'& \lambda'\end{array}\right)=\left( \begin{array}{cc} aa'+\phi(m\otimes n) &am'+\lambda'm\\ na'+\lambda n'& \lambda\lambda'\end{array}\right).$$
\begin{proposition}[\cite{Chen2}]\label{proposition homological epimorphism induces singularity equivalences}
Keep the notation and assumption as above. Then there is a triangle equivalence $D^b_{sg}(\Gamma)\simeq D^b_{sg}(A/\Im\phi)$.
\end{proposition}

Note that the above construction contains the one-point extension and one-point coextension of algebras, where $M$ or $N$ is zero.

The following two lemmas are crucial to the proof of our first main theorem.

\begin{lemma}\label{lemma 1}
Let $(Q,Sp,I)$ be a skewed-gentle triple and $(Q^{sg},I^{sg})$ its corresponding skewed-gentle pair. Then for any paths $p_1=\alpha_1\dots \alpha_n,p_2=\alpha_{n+1}\dots\alpha_{n+m}$ with $p_1,p_2\notin \langle I^{sg}\rangle$ and $\alpha_n\alpha_{n+1}\notin \langle I^{sg}\rangle$, we get that $p_1p_2\notin \langle I^{sg}\rangle$.
\end{lemma}
\begin{proof}

We denote this path as the following diagram shows:
\[\xymatrix{\circ^{a_1}  &\circ^{a_2}\ar[l]_{\alpha_1} &\cdots \ar[l]_{\alpha_2}& \circ^{a_{n+1}}\ar[l]_{\alpha_n} &\circ^{a_{n+2}}\ar[l]_{\alpha_{n+1}} &\cdots \ar[l]_{\alpha_{n+2}} &\circ^{a_{n+m+1}}\ar[l]_{\alpha_{n+m}}. }\]

Note that $p_1,p_2,\alpha_n\alpha_{n+1}\notin \langle I^{sg}\rangle$. Then we have $\alpha_i\alpha_{i+1}\notin I^{sg}$ for any $1\leq i\leq n+m-1$. First we define an operator $\Phi$ on $p_1p_2$ as follows: for any $2\leq s\leq n+m$, if $a_s$ comes from a special vertex $b_s\in Sp$, then there exists two arrows $\beta_1,\beta_2$ with $s(\beta_1)=t(\beta_2)$, such that $a_s\in\{b_s^+,b_s^-\}$ and $\{\alpha_{s-1},\alpha'_{s-1}\}=\{(a_s,\beta_1,a_{s-1}),(a'_s,\beta_1,a_{s-1})\}$, $\{\alpha_{s},\alpha'_{s}\}=\{(a_{s+1},\beta_2,a_{s}),(a_{s+1},\beta_2,a'_{s})\}$, where $a'_{s}$ is the vertex such that $\{a_s,a_s'\}=\{b_s^+,b_s^-\}$, by the definition of $I^{sg}$.
Set $p_1'=\alpha'_1\dots \alpha'_n$ and $p_2'=\alpha_{n+1}'\dots \alpha_{n+m}'$, where $\alpha_i'=\alpha_i$ for $i\notin\{ s-1,s\}$. Then $p_1p_2\sim p_1'p_2'$ and we define $\Phi(p_1p_2):=p_1'p_2'$.

If $\alpha'_{i}\alpha'_{i+1}\in I^{sg}$ for some $1\leq i\leq n+m-1$, then $i=s-2$, $s-1$ or $s$, since $\alpha'_{i}\alpha'_{i+1}=\alpha_i\alpha_{i+1}\notin I^{sg}$ for any $i\neq s-2,s-1,s$. However, it easy to see that $\alpha'_{s-1}\alpha_s'\notin I^{sg}$ from $\alpha'_{s-1}\alpha'_s\sim\alpha_{s-1}\alpha_s$. So $i=s-2$ or $s$, and we have $\alpha_{s-2}\alpha'_{s-1}\in I^{sg}$ or $\alpha'_{s}\alpha_{s+1}\in I^{sg}$, by the definition of $(Q^{sg},I^{sg})$, which implies $\alpha_{s-2}\alpha_{s-1}\in I^{sg}$ or $\alpha_{s}\alpha_{s+1}\in I^{sg}$ respectively, contradicts. So $\alpha'_{i}\alpha'_{i+1}\notin I^{sg}$ for any $1\leq i\leq n+m-1$.

Suppose $p_1p_2\in \langle I^{sg}\rangle$. Since $p_1,p_2,\alpha_n\alpha_{n+1}\notin \langle I^{sg}\rangle$ and the relations in $I^{sg}$ are either the zero relations or commutative relations, we know that there must be a finite sequence of operations:
$$p_1p_2\xrightarrow{\Phi_1}p_1'p_2'\xrightarrow{\Phi_2}\cdots \xrightarrow{\Phi_r}p_1^{(r)}p_2^{(r)},$$
where $p_1^{(i)}=\alpha_1^{(i)}\dots\alpha_n^{(i)}$ is a path of length $n$ and $p_2^{(i)}=\alpha_{n+1}^{(i)}\dots\alpha_{n+m}^{(i)}$ is a path of length $m$ for any $1\leq i\leq r$, and $\Phi_j$ is the operator defined for some special vertex in $p_1^{(j-1)}p_2^{(j-1)}$ for any $1\leq j\leq r$, such that $\alpha_k^{(r)}\alpha_{k+1}^{(r)}\in I^{sg}$ for some $1\leq k\leq n+m-1$. By the property of the operator $\Phi_1$ and $\alpha_i\alpha_{i+1}\notin I^{sg}$ for any $1\leq i\leq n+m-1$, we know that $\alpha'_{i}\alpha'_{i+1}\notin I^{sg}$ for any $1\leq i\leq n+m-1$, and discussing $\Phi_j$ recursively, we get that $\alpha_i^{(j)}\alpha_{i+1}^{(j)}\notin I^{sg}$ for any $1\leq j\leq r$ and $1\leq i\leq n+m-1$, which contradicts to $\alpha_k^{(r)}\alpha_{k+1}^{(r)}\in I^{sg}$ for some $1\leq k\leq n+m-1$. So $p_1p_2\notin \langle I^{sg}\rangle$.
\end{proof}

\begin{remark}
The above lemma is not true for any finite dimensional algebra $KQ/\langle I\rangle$ with the relations in $I$ being zero or commutative relations. Here is an example, let $(Q,I)$ be the pair as the following diagram shows.
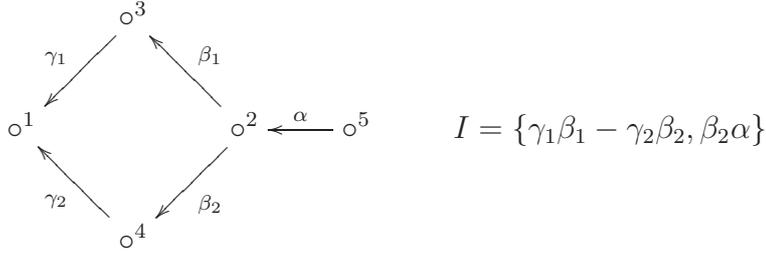
\[ \xymatrix{&\circ^{3}\ar[dl]_{\gamma_1}&&&\\
\circ^1 &  & \circ^2 \ar[ul]_{\beta_1} \ar[dl]^{\beta_2} & \circ^5\ar[l]_\alpha &I=\{ \gamma_1\beta_1-\gamma_2\beta_2,\beta_2\alpha\}\\
&\circ^{4}\ar[ul]^{\gamma_2} &&&
}\]
Then $\gamma_1\beta_1\notin \langle I\rangle,\beta_1\alpha\notin \langle I\rangle$, but $\gamma_1\beta_1\alpha\sim\gamma_2\beta_2\alpha\in \langle I\rangle$.
\end{remark}

\begin{lemma}\label{lemma 2}
Let $(Q,Sp,I)$ be a skewed-gentle triple and $(Q^{sg},I^{sg})$ its corresponding skewed-gentle pair. For any oriented cycle $\alpha_2\alpha_3\dots \alpha_n\alpha_1$ in $Q^{sg}$,
we get that either $\alpha_1\alpha_2\in \langle I^{sg}\rangle$, or $\alpha_2\alpha_3\dots \alpha_n\alpha_1\in \langle I^{sg}\rangle$.
\end{lemma}
\begin{proof}
Suppose $\alpha_1\alpha_2\notin \langle I^{sg}\rangle$ and $\alpha_2\alpha_3\dots \alpha_n\alpha_1\notin \langle I^{sg}\rangle$. Then Lemma \ref{lemma 1} implies
$$(\alpha_2\alpha_3\dots\alpha_n\alpha_1)^n\notin \langle I^{sg}\rangle,$$
for any $n>0$, and so $KQ^{sg}/\langle I^{sg}\rangle$ is infinite dimensional, a contradiction.
\end{proof}

Let $Q$ be any quiver. For any path $\alpha_1\alpha_2\dots\alpha_n$ in $Q$, we call a path $c=\alpha_1\alpha_2\dots\alpha_n$ \emph{passing through a vertex} $b$ if $b=t(\alpha_j)$ for some $2\leq j\leq n$.

The following theorem is the first main result of the paper.
\begin{theorem}\label{main theorem}
Let $(Q,Sp,I)$ be a skewed-gentle triple. Then its corresponding skewed-gentle algebra $KQ^{sg}/\langle I^{sg}\rangle$ is singularity equivalent to the gentle algebra $KQ/\langle I\rangle$.
\end{theorem}
\begin{proof}
For any vertex $a\in Sp$, there are two vertices $a^+,a^-$ in $Q^{sg}$. We denote by $e_{a^-}$ the primitive idempotent corresponding to the vertex $a^-$. Set $\Gamma=KQ^{sg}/\langle I^{sg}\rangle$, and $\Gamma'=\Gamma/\Gamma e_{a^-}\Gamma$, whose quiver is obtained from $Q^{sg}$ by removing the vertex $a^-$ and the adjacent arrows $\alpha^{+},\alpha^-$. Then $\Gamma'$ is a skewed-gentle algebra, in fact, $\Gamma'$ is the skewed-gentle
algebra with the corresponding skewed-gentle triple $(Q,Sp'=Sp\setminus \{a\},I)$. We discuss $\Gamma$ in the following three cases.

Case (a). If the valency of $a$ is $0$, then $KQ^{sg}/\langle I^{sg}\rangle=\Gamma' \oplus K$, and it is easy to see that
$D^b_{sg}(KQ^{sg}/\langle I^{sg}\rangle)\simeq D^b_{sg}(\Gamma')$.

Case (b). If the valency of $a$ is $1$, then $KQ^{sg}/\langle I^{sg}\rangle$ is a one-point extension or one-point coextension of $\Gamma'$, so $D^b_{sg}(KQ^{sg}/\langle I^{sg}\rangle)\simeq D^b_{sg}(\Gamma')$ by Proposition \ref{proposition homological epimorphism induces singularity equivalences}.

Case (c). If the valency of $a$ is $2$, then there exist only two arrows $\alpha_1, \alpha_2$ in $Q$ such that $s(\alpha_1)=a=t(\alpha_2)$, and $\alpha_1\alpha_2\in I$. Set $b=s(\alpha_2)$ and $c=t(\alpha_1)$. We do not exclude the case $b=c$. Then there are four cases.

Case (c1). If $b,c\notin Sp$, then the quiver of $Q^{sg}$ is as the picture Case (c1) in Figure 1 shows.
Then $\alpha_1^+\alpha_2^+-\alpha_1^-\alpha_2^-\in I^{sg}$.

We fix some notations.
Let $A=(1-e_{a^-})\Gamma(1-e_{a^-})$,
$$M:=\Span_K\{p| p\mbox{ is a path in}(Q^{sg},I^{sg}), \mbox{ and }p=p_1\alpha_1^-\mbox{ for some path }p_1\},$$
$$N:=\Span_K\{p| p\mbox{ is a path in}(Q^{sg},I^{sg}), \mbox{ and }p=\alpha_2^-p_2\mbox{ for some path }p_2\}.$$
Then $M$ is naturally a left $A$-module, and $N$ is naturally a right $A$-module. Note that $M$ and $N$ are finite dimensional vector spaces since $\Gamma$ is finite dimensional.
Furthermore, the $A\mbox{-}A$-bimodule morphism $\phi:M\otimes_K N\rightarrow A$ is defined as
$$\phi(p_1\alpha_1^-\otimes \alpha_2^-p_2)=p_1\alpha_1^-\alpha_2^-p_2.$$

We claim that $\phi$ is injective. 
We define 
\begin{eqnarray*}
T_1&:=&\{p| p\mbox{ is a path in}(Q^{sg},I^{sg}), \mbox{ and }p=p_1\alpha_1^-\mbox{ for some path }p_1 \\
&&\mbox{ which does not pass through } e^+ \mbox{ for any special vertex } e\},
\end{eqnarray*}
and
\begin{eqnarray*}
T_2&:=&\{p| p\mbox{ is a path in}(Q^{sg},I^{sg}), \mbox{ and }p=\alpha_2^-p_2\mbox{ for some path }p_2 \\
&&\mbox{ which does not pass through } e^+ \mbox{ for any special vertex } e\},
\end{eqnarray*}
where $\alpha_1^+=(a^+,\alpha_1,c)$ and $\alpha_2^+=(b,\alpha_2,a^+)$.
By the definition of $(Q^{sg},I^{sg})$, it is easy to see that
$M=\Span_KT_1$ and $N=\Span_KT_2$.

Let $I_1=\{\beta\alpha\in I| t(\alpha)\notin Sp\}$ and $\Lambda= KQ/ \langle I_1\rangle$.
We also define $$S_1:=\{p| p\mbox{ is a path in}(Q,I_1), \mbox{ and }p=p_1\alpha_1\mbox{ for some path }p_1\}, M':=\Span_K S_1,$$
and 
$$S_2:=\{p| p\mbox{ is a path in}(Q,I_1), \mbox{ and }p=\alpha_2 p_2\mbox{ for some path }p_2\}, N':=\Span_K S_2.$$
In fact, $M'$ (resp. $N'$) is a left (resp. right) $\Lambda$-module which is isomorphic to the radical of the indecomposable left (resp. right) projective module $P_{a}$ corresponding to the vertex $a$.
Note that $\Lambda$ is the algebra $KQ/\langle I_1\rangle $ with the ideal generated by zero-relations of length two. 
So $S_1$ (resp. $S_2$) is a basis of $M'$ (resp. $N'$). Denote by $S_1=\{u_1,\dots,u_m\}$, $S_2=\{v_1,\dots,v_n\}$.
By the definition of $(Q^{sg},I^{sg})$, we get that
$$\dim_K M=\dim_K M'+\sharp\{u_i|t(u_i)\in Sp\}=\sharp (T_1),$$ 
$$\dim_K N=\dim_K N'+\sharp\{v_i|s(v_i)\in Sp \}=\sharp (T_2).$$
which yields that
$T_1$ and $T_2$ are basis of the linear spaces $M$ and $N$ over $K$, respectively.

Similarly, we get that \begin{eqnarray*}
T_A&:=&\{p| p\mbox{ is a path in}(Q^{sg},I^{sg}), \mbox{ and }p
\mbox{ does not pass through }\\
&& e^+ \mbox{ for any special vertex } e, \mbox{ and }s(p),t(p)\neq a^- \}
\end{eqnarray*} is a basis of $A$ over $K$.
Note that $$\phi(p_1\alpha_1^-\otimes \alpha_2^-p_2)=p_1\alpha_1^-\alpha_2^-p_2,$$ for any $p_1\alpha_1^-\in T_1$ and $\alpha_2^-p_2\in T_2$. Denote by $T:=\{u_i\otimes v_j| u_i\in T_1, v_j\in T_2\}$. Then $T$ is a basis of $M\otimes _K N$.
Lemma \ref{lemma 1} yields that $p_1\alpha_1^-\alpha_2^-p_2$ is nonzero in $A$, which is in $T_A$. It is easy to see that $\phi$ induces an injective map from $T$ to $T_A$, which means $\phi$ is injective.

It follows from Lemma \ref{lemma 2} that $e_{a^-}\Gamma e_{a^-}$ is isomorphic to $K$.
We identify $\Gamma$ with $\left( \begin{array}{cc} A &M \\ N& K \end{array}  \right)$, where the $K$ in the southeast corner is identified with $e_{a^-}\Gamma e_{a^-}$.
Note that $A/\Im \phi=\Gamma'$. Then Proposition \ref{proposition homological epimorphism induces singularity equivalences} yields a triangle equivalence
$$D^b_{sg}(\Gamma)\simeq D^b_{sg}(\Gamma').$$

Case (c2). If $b\in Sp$, $c\notin Sp$, then the quiver $Q^{sg}$ is as the picture Case (c2) in Figure 1 shows.
Then $\gamma_1\beta_1-\gamma_2\beta_2$, $\gamma_1\beta_3-\gamma_2\beta_4\in I^{sg}$. Let $A=(1-e_{a^-})\Gamma(1-e_{a^-})$,
$$M:=\Span_K\{p| p\mbox{ is a path in}(Q^{sg},I^{sg}), \mbox{ and }p=p_1\gamma_2\mbox{ for some path }p_1\},$$
$$N:=\Span_K\{p| p\mbox{ is a path in}(Q^{sg},I^{sg}), \mbox{ and }p=\beta_2p_2\mbox{ or }\beta_4p_2\mbox{ for some path }p_2\}.$$
The remaining is similar to Case (c1), we omit the proof.

Case (c3). If $b\notin Sp$, $c\in Sp$, then the quiver $Q^{sg}$ is as the picture Case (c3) in Figure 1 shows. This case is similar to Case (c2), we omit the proof.

Case (c4). If $b,c\in Sp$, then the quiver $Q^{sg}$ is as the picture Case (c4) in Figure 1 shows.
Then $\gamma_1\beta_1-\gamma_2\beta_2$, $\gamma_1\beta_3-\gamma_2\beta_4$, $\gamma_3\beta_1-\gamma_4\beta_2$, $\gamma_3\beta_3-\gamma_4\beta_4\in I^{sg}$. Let $A=(1-e_{a^-})\Gamma(1-e_{a^-})$,
$$M:=\Span_K\{p| p\mbox{ is a path in}(Q^{sg},I^{sg}), \mbox{ and }p=p_1\gamma_2\mbox{ or }p_1\gamma_4\mbox{ for some path }p_1\},$$
$$N:=\Span_K\{p| p\mbox{ is a path in}(Q^{sg},I^{sg}), \mbox{ and }p=\beta_2p_2\mbox{ or }\beta_4p_2\mbox{ for some path }p_2\}.$$
The remaining is similar to Case (c1), we omit the proof.

\setlength{\unitlength}{0.7mm}
\begin{center}
\begin{picture}(200,100)
\put(-10,70){\begin{picture}(100,30 )

\put(50,15){\circle*{2}}
\put(50,-15){\circle*{2}}
\qbezier(65,10)(75,0)(65,-10)
\qbezier(35,10)(25,0)(35,-10)
\qbezier(35,10)(50,30)(65,10)
\qbezier(35,-10)(50,-30)(65,-10)
\put(20,0){\circle*{2}}

\put(80,0){\circle*{2}}

\put(50,15){\vector(2,-1){29}}
\put(80,0){\vector(-2,-1){29}}

\put(50,15){\vector(-2,-1){29}}
\put(20,0){\vector(2,-1){29}}
\put(47,15){$b$}
\put(47,-19){$c$}
\put(12,-2){$a^+$}
\put(82,-2){$a^-$}
\put(30,10){$\alpha_2^+$}
\put(30,-12){$\alpha_1^+$}

\put(65,8){$\alpha_2^-$}
\put(65,-10){$\alpha_1^-$}
\put(47,-2){$Q'$}

\put(35,-35){Case (c1)}
\end{picture}}

\put(100,70){\begin{picture}(100,30)

\put(50,15){\circle*{2}}
\put(50,-15){\circle*{2}}
\put(50,10){\circle*{2}}
\qbezier(65,10)(75,0)(65,-10)
\qbezier(35,10)(25,0)(35,-10)
\qbezier(35,10)(50,30)(65,10)
\qbezier(35,-10)(50,-30)(65,-10)
\put(20,0){\circle*{2}}

\put(80,0){\circle*{2}}
\put(50,10){\vector(3,-1){29}}
\put(50,10){\vector(-3,-1){29}}
\put(50,15){\vector(2,-1){29}}
\put(80,0){\vector(-2,-1){29}}

\put(50,15){\vector(-2,-1){29}}
\put(20,0){\vector(2,-1){29}}
\put(47,15){$b^+$}
\put(47,5){$b^-$}
\put(47,-19){$c$}
\put(12,-2){$a^+$}
\put(82,-2){$a^-$}
\put(40,12){$\beta_1$}
\put(40,-12){$\gamma_1$}
\put(40,3){$\beta_3$}
\put(60,2){$\beta_4$}
\put(60,10){$\beta_2$}
\put(60,-10){$\gamma_2$}
\put(47,-2){$Q'$}

\put(35,-35){Case (c2)}
\end{picture}}

\put(-10,0){\begin{picture}(100,30)

\put(50,15){\circle*{2}}
\put(50,-15){\circle*{2}}
\put(50,-10){\circle*{2}}
\qbezier(65,10)(75,0)(65,-10)
\qbezier(35,10)(25,0)(35,-10)
\qbezier(35,10)(50,30)(65,10)
\qbezier(35,-10)(50,-30)(65,-10)
\put(20,0){\circle*{2}}

\put(80,0){\circle*{2}}
\put(20,0){\vector(3,-1){29}}
\put(80,0){\vector(-3,-1){29}}
\put(50,15){\vector(2,-1){29}}
\put(80,0){\vector(-2,-1){29}}

\put(50,15){\vector(-2,-1){29}}
\put(20,0){\vector(2,-1){29}}
\put(47,15){$b$}
\put(47,-9){$c^-$}
\put(47,-19){$c^+$}
\put(12,-2){$a^+$}
\put(82,-2){$a^-$}
\put(40,10){$\beta_1$}
\put(40,-15){$\gamma_1$}
\put(40,-5){$\gamma_3$}
\put(60,-5){$\gamma_4$}
\put(60,8){$\beta_2$}
\put(60,-12){$\gamma_2$}
\put(47,-2){$Q'$}

\put(35,-35){Case (c3)}
\end{picture}}

\put(100,0){\begin{picture}(100,30)

\put(50,15){\circle*{2}}
\put(50,-15){\circle*{2}}
\put(50,-10){\circle*{2}}
\qbezier(65,10)(75,0)(65,-10)
\qbezier(35,10)(25,0)(35,-10)
\qbezier(35,10)(50,30)(65,10)
\qbezier(35,-10)(50,-30)(65,-10)
\put(20,0){\circle*{2}}

\put(80,0){\circle*{2}}
\put(20,0){\vector(3,-1){29}}
\put(80,0){\vector(-3,-1){29}}
\put(50,15){\vector(2,-1){29}}
\put(80,0){\vector(-2,-1){29}}

\put(50,15){\vector(-2,-1){29}}
\put(20,0){\vector(2,-1){29}}
\put(47,15){$b^+$}
\put(47,-9){$c^-$}
\put(47,-19){$c^+$}
\put(12,-2){$a^+$}
\put(82,-2){$a^-$}
\put(40,10){$\beta_1$}
\put(40,-15){$\gamma_1$}
\put(40,-5){$\gamma_3$}
\put(60,-5){$\gamma_4$}
\put(60,8){$\beta_2$}
\put(60,-12){$\gamma_2$}
\put(47,-2){$Q'$}
\put(50,10){\circle*{2}}
\put(40,3){$\beta_3$}
\put(60,2){$\beta_4$}
\put(47,5){$b^-$}
\put(50,10){\vector(3,-1){29}}
\put(50,10){\vector(-3,-1){29}}
\put(35,-35){Case (c4)}
\end{picture}}
\put(50,-50){Figure 1. The quiver $Q^{sg}$ in Case (c).}
\end{picture}
\vspace{4cm}
\end{center}

To sum up, we get that $D^b_{sg}(\Gamma)\simeq D^b_{sg}(\Gamma')$.
Since $\Gamma'$ is also skewed-gentle, we replace $\Gamma$ in the above with $\Gamma'$, and discuss it recursively. After $\sharp(Sp)$ steps, we can get that $D^b_{sg}(\Gamma)\simeq D^b_{sg}(KQ/\langle I\rangle)$. The proof is complete.
\end{proof}

\begin{corollary}\label{corollary 1}
Let $(Q,Sp,I)$ be a skewed-gentle triple. Then the following statements are equivalent.

(i) $\gldim KQ^{sg}/\langle I^{sg}\rangle<\infty,$

(ii) $\gldim KQ/\langle I\rangle<\infty.$
\end{corollary}
\begin{proof}
It follows from Theorem \ref{main theorem} that $\gldim KQ^{sg}/\langle I^{sg}\rangle<\infty$ if and only if $\gldim KQ/\langle I\rangle<\infty$ since they are singularity equivalent.
\end{proof}

\section{The second main theorem}
We first recall the singularity category of a gentle algebra due to \cite{Ka}. For a gentle algebra $\Lambda=KQ/I$, we denote by $\cc(\Lambda)$ the set of equivalence classes (with respect to cyclic permutation) of \emph{repetition-free} cyclic paths $\alpha_1\dots\alpha_n$ in $Q$ such that $\alpha_i\alpha_{i+1}\in I$ for all $i$, where we set $n+1=1$. For convenience, we call any element $c$ in $\cc(\Lambda)$ to be \emph{full repetition-free cyclic}. For every arrow $\alpha\in Q_1$, there is at most one cycle $c\in\cc(\Lambda)$ containing it. Moreover, we set $l(c)$ for the \emph{length} of a cycle $c\in\cc(\Lambda)$, i.e. $l(\alpha_1\dots\alpha_n)=n$.
\begin{theorem}[\cite{Ka}]\label{theorem Kalck}
There is an equivalence of triangulated categories
$$D^b_{sg}(\Lambda)\simeq \prod_{c\in\cc(\Lambda)} \frac{D^b( K)}{[l(c)]} ,$$
where $ D^b(K)/[l(c)]$ denotes the triangulated orbit category, see Keller \cite{Ke}.
\end{theorem}

Let $(Q,Sp,I)$ be a skewed-gentle triple. For any $c=\alpha_1\dots\alpha_n\in\cc(\Lambda)$, if there are even (resp., odd) number of special vertices lying on $c$, then we call $c$ an \emph{even (resp., odd) repetition-free cyclic path}. We denote by $\cc^{even}(\Lambda)$ (resp., $\cc^{odd}(\Lambda)$) the set of even (resp., odd) repetition-free cyclic path.
Recall that we call a cyclic path $c=\alpha_1\alpha_2\dots\alpha_n$ passing through a vertex $b$ if $b=t(\alpha_j)$ for some $2\leq j\leq n$.

For any path $c=\alpha_1\dots\alpha_n$ in $Q$ and $2\leq i\leq n$, we set
$$\sigma_i(c):=\left\{ \begin{array}{ccc} + &\mbox{if }\alpha_1\dots\alpha_i\mbox{ passes through even number of special vertices,}\\
- &\mbox{if }\alpha_1\dots\alpha_i\mbox{ passes through odd number of special vertices.}    \end{array}  \right.$$
Sometimes, we write $\sigma_i(c)$ to be $\sigma_i$ if $c$ is obvious.
Similarly, for any $2\leq i\leq n$, we set
$$\tau_i(c):=\left\{ \begin{array}{ccc} + &\mbox{if }\alpha_1\dots\alpha_i\mbox{ passes through odd number of special vertices,}\\
- &\mbox{if }\alpha_1\dots\alpha_i\mbox{ passes through even number of special vertices.}    \end{array}  \right.$$

\begin{lemma}\label{lemma 3}
Let $(Q,Sp,I)$ be a skewed-gentle triple. For any oriented path $c=\alpha_1\alpha_2\dots\alpha_n$ in $Q$, we get that $\alpha_1^+\alpha_2^{\sigma_2}\dots\alpha_n^{\sigma_n}$ and $\alpha_1^-\alpha_2^{\tau_2}\dots\alpha_n^{\tau_n}$ is an oriented path in $Q^g$. In particular, if $\alpha_{i-1}\alpha_i\in I$ for some $2\leq i\leq n$, then $\alpha_{i-1}^{\sigma_{i-1}}\alpha_i^{\sigma_i}$ and $\alpha_{i-1}^{\tau_{i-1}}\alpha_i^{\tau_i}$ are in $I^g$.
\end{lemma}

\begin{proof}
We only need prove it for $\alpha_1^+\alpha_2^{\sigma_2}\dots\alpha_n^{\sigma_n}$, the other is similar.
Set $\sigma_1=+$. For any $2\leq i\leq n$, if $t(\alpha_i)$ is special, then $s(\alpha_{i-1}^\pm)=s(\alpha_i)=t(\alpha_i^\pm)$. On the other hand, if
$t(\alpha_i)$ is ordinary, then $\sigma_{i-1}=\sigma_i$, so $s(\alpha_{i-1}^{\sigma_{i-1}})=s(\alpha_i)^{\sigma_{i-1}}=t(\alpha_i^{\sigma_{i}})$.
To sum up, we get that $s(\alpha_{i-1}^{\sigma_{i-1}})=t(\alpha_i^{\sigma_{i}})$ for any $2\leq i\leq n$, so $\alpha_1^+\alpha_2^{\sigma_2}\dots\alpha_n^{\sigma_n}$
is a path in $Q^g$.

For the last statement, if $t(\alpha_i)$ is special, then $\sigma_{i-1}\neq \sigma_i$, which implies that $\alpha_{i-1}^{\sigma_{i-1}}\alpha_i^{\sigma_i}\in I^g$ by the definition of $I^g$; if $t(\alpha_i)$ is ordinary, then $\sigma_{i-1}= \sigma_i$, which also implies that $\alpha_{i-1}^{\sigma_{i-1}}\alpha_i^{\sigma_i}\in I^g$.
\end{proof}

\begin{lemma}\label{lemma 4}
Let $(Q,Sp,I)$ be a skewed-gentle triple. Then for any arrow $\alpha_1$ in $Q$, the following statements are equivalent.

(i) There is a full repetition-free cyclic path in $Q$ containing $\alpha_1$,

(ii) There is a full repetition-free cyclic path in $Q^g$ containing $\alpha_1^+$,

(iii) There is a full repetition-free cyclic path in $Q^g$ containing $\alpha_1^-$.
\end{lemma}
\begin{proof}
Recall that $Q_0^g:=\cup_{i\in Q_0}Q_0[i]$ where $Q_0[i]$ is the set $\{i\}$ (resp., $\{i^-,i^+\}$) if the vertex $i$ is special (resp., ordinary), $Q_1^g:=\{\alpha^+,\alpha^-|\alpha\in Q_1\}$,
$$s(\alpha^{\pm}):=\left\{ \begin{array}{cc} s(\alpha)^{\pm}, &\mbox{if } s(\alpha)\notin Sp;\\
s(\alpha), &\mbox{if } s(\alpha)\in Sp,\end{array}
 \right.\quad\quad t(\alpha^{\pm}):=\left\{ \begin{array}{cc} t(\alpha)^{\pm}, &\mbox{if } t(\alpha)\notin Sp;\\
t(\alpha), &\mbox{if } t(\alpha)\in Sp
 \end{array}\right.$$
and $$I^g:=\{\beta^+\alpha^+,\beta^-\alpha^-|\beta\alpha\in I,t(\alpha)\notin Sp\}\bigcup \{\beta^+\alpha^-,\beta^-\alpha^+|\beta\alpha\in I,t(\alpha)\in Sp\}.$$

For $(i)\Rightarrow(ii),(iii)$, let $c=\alpha_1\alpha_2\dots\alpha_n$ be a full repetition-free cyclic path in $Q$ containing $\alpha_1$. We discuss it in the following two cases.

(a) If $c\in \cc^{even}(\Lambda)$, then we claim that $\alpha_1^{+}\alpha_2^{\sigma_2}\dots \alpha_n^{\sigma_n}$ and $\alpha_1^-\alpha_2^{\tau_2}\dots \alpha_n^{\tau_n}$ are full repetition-free cyclic paths.

We only need prove that for $\alpha_1^{+}\alpha_2^{\sigma_2}\dots \alpha_n^{\sigma_n}$, the other is similar. In fact, Lemma \ref{lemma 3} yields that it is an oriented path in $Q^g$.

If $t(\alpha_1)=s(\alpha_n)$ is special, then
$t(\alpha_1^{\pm})=s(\alpha_n^{\pm})=t(\alpha_1)$, so $\alpha_1^{+}\alpha_2^{\sigma_2}\dots \alpha_n^{\sigma_n}$ is a cyclic path. Furthermore,
$\alpha_1\dots\alpha_n$ passes through odd number of special vertices, so $\sigma_n=-$, which implies $\alpha_n^{\sigma_n}\alpha_1^+\in I^g$.
Lemma \ref{lemma 3} also shows that
$$\alpha_1^+\alpha_2^{\sigma_2},\dots,\alpha_{i-1}^{\sigma_{i-1}}\alpha_i^{\sigma_i},\dots,\alpha_{n-1}^{\sigma_{n-1}}\alpha_n^{\sigma_n}\in I^g,$$
and then together with above, we get that $\alpha_1^{+}\alpha_2^{\sigma_2}\dots \alpha_n^{\sigma_n}$ is full repetition-free.

If $t(\alpha)=s(\alpha_n)$ is ordinary, then $t(\alpha^+)=t(\alpha)^+=s(\alpha_n)^+$ and $t(\alpha^-)=t(\alpha)^-=s(\alpha_n)^-$. In particular, $\alpha_1\dots\alpha_n$ passes through even number of special vertices, so $\sigma_n=+$, which implies $\alpha_1^{+}\alpha_2^{\sigma_2}\dots \alpha_n^{\sigma_n}$ is cyclic and $\alpha_n^{\sigma_n}\alpha_1^+\in I^g$.
Similarly, we get that it is also full repetition-free.

(b) If $c\in \cc^{odd}(\Lambda)$, set $l=\alpha_1\dots\alpha_n\alpha_1\dots\alpha_n$, then we claim that
$$p=\alpha_1^{+}\alpha_2^{\sigma_2(l)}\dots \alpha_n^{\sigma_n(l)}\alpha_1^{\sigma_{n+1}(l)}\alpha_{2}^{\sigma_{n+2}(l)}\dots \alpha_{n}^{\sigma_{2n}(l)}$$
is a full repetition-free cyclic path. From $c\in\cc^{odd}(\Lambda)$, it is easy to see that $\sigma_{n+1}(l)=-$ and $\sigma_{n+i}(l)=\tau_i(c)$ for any $2\leq i\leq n$. Thus
$$p=\alpha_1^{+}\alpha_2^{\sigma_2(c)}\dots \alpha_n^{\sigma_n(c)}\alpha_1^{-}\alpha_2^{\tau_2(c)}\dots \alpha_n^{\tau_n(c)}.$$
Similar to (a), we only need check that $s(\alpha_n^{\sigma_{2n}(l)})= t(\alpha_1^+)$ and $\alpha_n^{\sigma_{2n}(l)}\alpha_1^+\in I^g$.
If $s(\alpha_n)$ is special, then $l$ passes through odd number of special vertices, and so $\sigma_{2n}(l)=-$. In this case, we have
$s(\alpha_n^\pm)=t(\alpha_1^\pm)$, and then $\alpha_n^{\sigma_{2n}(l)}\alpha_1^+\in I^g$ by the definition of $I^g$. If $s(\alpha_n)$ is ordinary, then $l$ passes through even number of special vertices, and so $\sigma_{2n}(l)=+$. In this case, we also have
$s(\alpha_n^+)=s(\alpha_n)^+=t(\alpha_1^+)$, and then $\alpha_n^{\sigma_{2n}(l)}\alpha_1^+\in I^g$ by the definition of $I^g$.

For $(ii)\Rightarrow(i)$ and $(iii)\Rightarrow(i)$, let $\alpha_1^{\delta_1}\alpha_2^{\delta_2}\dots\alpha_n^{\delta_n}$ be any full repetition-free cyclic path in $Q^g$, where $\delta_i=+$ or $-$, and $\alpha_i$ is an arrow in $Q$ for any $1\leq i\leq n$. Since $s(\alpha_{i}^{\delta_i})=t(\alpha_{i+1}^{\delta_{i+1}})$ for any $1\leq i\leq n$, where we set $n+1=1$, it is easy to see that $s(\alpha_i)=t(\alpha_{i+1})$. So $\alpha_1\alpha_2\dots\alpha_n$ is cyclic. On the other hand, by the definition of $I^g$, every zero relation in $I^g$ comes from a zero relation in $I$, it is easy to see that $\alpha_i\alpha_{i+1}\in I$ since $\alpha_i^{\delta_i}\alpha_{i+1}^{\delta_{i+1}}\in I^g$ for any $1\leq i\leq n$. If $\alpha_i\neq \alpha_j$ for any $i\neq j$, then $\alpha_1\alpha_2\dots\alpha_n$ is a repetition-free cyclic path.

Otherwise, without losing generality, we assume that $\alpha_{m+1}=\alpha_1$ for some $1<m< n$.
It is easy to see that $\delta_{m+1}\neq \delta_1$ since $\alpha_1^{\delta_1}\alpha_2^{\delta_2}\dots\alpha_n^{\delta_n}$ is repetition-free, which also implies that $m$ is unique. In fact, for any $1\leq i\leq n$, there exists at most one $j$ satisfying $1\leq j\neq i\leq n$ and $\alpha_i=\alpha_j$.
We claim that $\alpha_1\alpha_2\dots\alpha_m$ is a full repetition-free cyclic path.
Since $(Q,I)$ is gentle, for $\alpha_1$, there exists at most one arrow $\beta$ with $t(\beta)=s(\alpha_1)$ such that $\alpha_1\beta\in I$. However, we know that $\alpha_1\alpha_2,\alpha_{m+1}\alpha_{m+2}\in I$, so $\alpha_{m+2}=\alpha_2$. Inductively, we can get that $\alpha_{m+i}=\alpha_{i}$ for $1\leq i\leq m$. Since $\alpha_n\alpha_1\in I$ and $\alpha_{m}\alpha_{m+1}\in I$, we also get that $\alpha_n=\alpha_{m}=\alpha_{2m}$, which yields $\alpha_1\alpha_2\dots\alpha_m$ is cyclic and $\alpha_1\alpha_2\dots\alpha_m=\alpha_{m+1}\alpha_{m+2}\dots\alpha_{2m}$. In fact, we also know $n=2m$. For any $1\leq i\neq j\leq m$, if $\alpha_i=\alpha_j$, then $\alpha_{m+i}=\alpha_i=\alpha_j=\alpha_{m+j}$, which contradicts to $\alpha_1^{\delta_1}\alpha_2^{\delta_2}\dots\alpha_n^{\delta_n}$ is repetition-free. So $\alpha_1\alpha_2\dots\alpha_m$ is repetition-free.
To sum up, we get that $\alpha_1\alpha_2\dots\alpha_m$ is a full repetition-free cyclic path.
\end{proof}

The following theorem is the second main result of this paper.
\begin{theorem}\label{main theorem 2}
Let $(Q,Sp,I)$ be a skewed-gentle triple.
There is an equivalence of triangulated categories
$$D^b_{sg}(KQ^g/\langle I^g\rangle)\simeq (\prod_{c\in\cc^{even}(\Lambda)} \frac{D^b(K)}{[l(c)]})\prod (\prod_{c\in\cc^{even}(\Lambda)} \frac{D^b(K)}{[l(c)]})\prod(\prod_{c\in\cc^{odd}(\Lambda)} \frac{D^b( K)}{[2l(c)]} ),$$
where $\Lambda=KQ/\langle I\rangle$ and $ D^b(K)/[l(c)]$ denotes the triangulated orbit category, see Keller \cite{Ke}.
\end{theorem}
\begin{proof}
Since $(Q^g, I^g)$ is a gentle pair, for any arrow $\alpha^{\delta}\in Q^g$, where $\delta=+$ or $-$, there is at most one full repetition-free cyclic path containing $\alpha^\delta$.
It follows from Lemma \ref{lemma 4} and its proof that
\begin{eqnarray*}
\cc(KQ^g/\langle I^g\rangle)=\{ \alpha_1^{+}\alpha_2^{\sigma_2}\dots \alpha_n^{\sigma_n},\alpha_1^-\alpha_2^{\tau_2}\dots \alpha_n^{\tau_n} |c=\alpha_1\alpha_2\dots \alpha_n\in\cc^{even}(\Lambda)\}\\
\bigcup \{ \alpha_1^{+}\alpha_2^{\sigma_2(c)}\dots \alpha_n^{\sigma_n(c)}\alpha_1^{-}\alpha_2^{\tau_2(c)}\dots \alpha_n^{\tau_n(c)}|c=\alpha_1\alpha_2\dots \alpha_n\in\cc^{odd}(\Lambda)\}.
\end{eqnarray*}
Then Theorem \ref{theorem Kalck} yields the result immediately.
\end{proof}

In general, we do not have $D^b_{sg}(KQ^{sg}/\langle I^{sg}\rangle)\simeq D^b_{sg}(KQ^{g}/I^{g})$. In fact, Theorem \ref{main theorem 2} shows that $D^b_{sg}(KQ^{sg}/\langle I^{sg}\rangle)\simeq D^b_{sg}(KQ^{g}/I^{g})$
if and only if they are zero. So we have the following direct corollary.

\begin{corollary}\label{corollary 2}
Let $(Q,Sp,I)$ be a skewed-gentle triple. Then the following statements are equivalent.

(i) $\gldim KQ^{sg}/\langle I^{sg}\rangle<\infty$,

(ii) $\gldim KQ/\langle I\rangle<\infty$,

(iii) $\gldim KQ^g/\langle I^g\rangle<\infty$.
\end{corollary}
\begin{proof}
By Corollary \ref{corollary 1}, we only need prove that $(ii)\Leftrightarrow(iii)$. Theorem \ref{main theorem 2} shows that $D^b_{sg}(KQ/\langle I\rangle)\simeq0$ if and only if $D^b_{sg}(KQ^g/\langle I^g\rangle)\simeq0$, which implies that $\gldim KQ/\langle I\rangle<\infty$ if and only if $\gldim KQ^g/\langle I^g\rangle<\infty$.
\end{proof}
\begin{remark}
Note that we do not assume $\Char K\neq 2$ in Corollary \ref{corollary 2}. When $\Char K\neq 2$, we know that $KQ^{sg}/\langle I^{sg}\rangle$ is Morita equivalent to a skew-group algebra $(KQ^g/\langle I^g\rangle)G$ defined in Section 2.3. In this case, \cite[Theorem 1.3]{RR} shows that $\gldim KQ^{sg}/\langle I^{sg}\rangle<\infty$ if and only if $\gldim KQ^{g}/\langle I^g\rangle<\infty$.
\end{remark}

In the following, we denote by $S_n$ the self-injective Nakayama algebra of a cyclic quiver with $n$ vertices modulo the ideal generated by paths of length $2$.

\begin{example}
(a) Keep the notations as in Example \ref{example 2} (a) and Example \ref{example 3} (a). Then $D^b_{sg}(KQ^{sg}/\langle I^{sg}\rangle)\simeq \underline{\mod}(KQ/\langle I\rangle)\simeq \underline{\mod}S_2$, while $D^b_{sg}(KQ^g/\langle I^g\rangle)\simeq \underline{\mod}S_4$.

(b) Keep the notations as in Example \ref{example 2} (b) and Example \ref{example 3} (b). Then $D^b_{sg}(KQ^{sg}/\langle I^{sg}\rangle)\simeq \underline{\mod}(KQ/\langle I\rangle)\simeq \underline{\mod}S_3$, while $D^b_{sg}(KQ^g/\langle I^g\rangle)\simeq \underline{\mod}S_6$.
\end{example}


\begin{thebibliography}{AAA}
\bibitem{ABCP} I. Assem, T. Br\"{u}stle, G. Charbonneau-Jodoin and P. Plamondon, Gentle algebras arising from surface
triangulations. Algebra Number Theory 4(2)(2010), 201-229.
\bibitem{AS} I. Assem and A. Skowro\'{n}ski, Iterated tilted algebras of type $\tilde{\A}_n$. Math. Z. 195(1987), 269-290.
\bibitem{AR1} M. Auslander and I. Reiten, Application of contravariantly finite subcategories. Adv. Math. 86(1)(1991), 111-152.
\bibitem{AR2} M. Auslander and I. Reiten, Cohen-Macaulay and Gorenstein artin algebras, In: Representation Theory of finite Groups and Finite-Dimensional Algebras, Progress in Math. 95, Birkh\"{a}user Verlag, Basel, 1991, 221-245.
\bibitem{AM} L. L. Avramov and A. Martsinkovsky, Absolute, relative and Tate cohomology of modules of finite Gorenstein dimensions. Proc. London Math. Soc. 85(3)(2002), 393-440.
\bibitem{BM} V. Bekkert and H. A. Merklen, Indecomposables in derived categories of gentle alegbras. Algebr. Represent. Theor. 6(2003), 285-302.



\bibitem{BMM} V. Bekkert, E. N. Marcos and H. A. Merklen, Indecomposables in derived categories of skewed-gentle alegbras. Comm. Algebra 31(6)(2003), 2615-2654.
\bibitem{BH} C. Bessenrodt and T. Holm, $q$-Cartan matrices and combinatorial invariants of derived categories fro skewed-gentle algebras. Pac. J. Math. 229(1)(2007), 25-48.

\bibitem{Bu} R. Buchweitz, Maximal Cohen-Macaulay modules and Tate cohomology over Gorenstein Rings. Unpublished Manuscript, 1987. Available at Http://hdl.handle.net/1807/16682.
\bibitem{Bur} I. Burban, Derived categories of coherent sheaves on rational singular curves, In: Representations of
finite dimensional algebras and related topics in Lie Theory and geometry, Fields Inst. Commun. 40, Amer. Math. Soc., Providence, RI (2004), 173-188.
\bibitem{Chen1} X-W. Chen, Singularity categories, Schur functors and triangular matrix rings. Algebr. Represent. Theor. 12(2009), 181-191.
\bibitem{Chen2} X-W. Chen, Singular equivalences induced by homological epimorphisms. Proc. Amer. Math. Soc. 142(8)(2014), 2633-2640.
\bibitem{EJ} E. E. Enochs and O. M. G. Jenda, Relative homological algebra. de Gruyter Exp. Math. 30, Walter de Gruyter Co., 2000.

\bibitem{GdlP} C. Gei{\ss} and J. A. de la Pe\~{na}, Auslander-Reiten components for clans. Boll. Soc. Mat. Mexicana 5(1999), 307-326.
\bibitem{GR} C. Gei{\ss} and I. Reiten, Gentle algebras are Gorenstein. in Representations of algebras and related topics,
Fields Inst. Commun. 45, Amer. Math. Soc., Providence, RI (2005), 129-133.
\bibitem{Ha1} D. Happel, On gorenstein algebras, In: Representation Theory of finite Groups and Finite-Dimensional Algebras, Progress in Math. 95, Birkh\"{a}user Verlag, Basel, 1991, 389-404.
\bibitem{Ka} M. Kalck, Singularity categories of gentle algebras. Bull. London Math. Soc. 47(1)(2015), 65-74.
\bibitem{Ke} B. Keller, On triangulated orbit categories. Doc. Math. 10(2005), 551-581.
\bibitem{KN} S. Koenig and H. Nagase, Hochschild cohomology and stratifying ideals. J. Pure Appl. Algebra 213 (2009), 886-891.
\bibitem{Or1} D. Orlov, Triangulated categories of singularities and D-branes in Landau-Ginzburg models. Proc. Steklov Inst. Math. 246(3)(2004), 227-248.
\bibitem{RR} I. Reiten and Ch. Riedtmann, Skew group algebras in the representation
theory of Artin algebras. J. Algebra 92(1)(1985), 224-282.
\bibitem{Ri} J. Rickard, Derived categories and stable equivalences. J. Pure Appl. Algebra 61(1989), 303-317.
\bibitem{SW} A. Skowro\'{n}ski and J. Waschb\"{u}sch, Representation-finite biserial algebras. J. Reine Angew. Math. 345(1983), 172-181.
\end{thebibliography}
\end{document}